\documentclass[12pt,a4paper]{article}
\usepackage{enumitem}
\usepackage{amsmath}
\usepackage{amssymb}
\usepackage[english]{babel}
\usepackage[utf8]{inputenc}
\usepackage{amsfonts}
\usepackage{amsthm}
\usepackage{color}
\usepackage{authblk}
\usepackage{graphicx}
\usepackage{txfonts}
\usepackage{blindtext}
\providecommand{\keywords}[1]
{
  	
  \textbf{\textit{Keywords---}} #1
}
\providecommand{\chapter}[1]
{
  	
  \textbf{\textit{Chapter---}} #1
}
\newtheorem{thm}{Theorem}[section]
\newtheorem{prop}[thm]{Proposition}
\newtheorem{cor}[thm]{Corollary}
\newtheorem{lem}[thm]{Lemma}
\theoremstyle{remark}
\newtheorem{rmk}[thm]{Remark}
\theoremstyle{definition}
\newtheorem{defn}{Definition}[section]
\newtheorem*{assumptionH}{Assumption H}

\DeclareMathOperator{\N}{\mathbb{N}}
\DeclareMathOperator{\R}{\mathbb{R}}

\DeclareMathOperator{\diam}{diam}

\DeclareMathOperator{\cH}{\mathcal{H}}

\DeclareMathOperator{\cD}{\mathcal{D}}
\DeclareMathOperator{\cL}{\mathcal{L}}
\DeclareMathOperator{\cB}{\mathcal{B}}

\newcommand{\Norm}[2]{\left\Vert #1 \right\Vert_{#2}}

\begin{document}
\title{Duality and distance formulas in Lipschitz - H\"{o}lder spaces}

\author[1]{Francesca Angrisani \thanks{francesca.angrisani@unina.it}}
\author[1]{Giacomo Ascione \thanks{giacomo.ascione@unina.it}}
\author[2]{Luigi D'Onofrio \thanks{Corresponding Author: donofrio@uniparthenope.it}}
\author[1]{Gianluigi Manzo\thanks{gianluigi.manzo@unina.it}}
\affil[1]{{Dipartimento di Matematica ed Applicazioni \lq\lq R. Caccioppoli\rq\rq   Universit\'a di Napoli \lq\lq Federico II\rq\rq}}
\affil[2]{Dipartimento di Scienze e Tecnologie, Universit\'a di Napoli \lq\lq Parthenope \rq\rq }
\date{} 
\maketitle
\begin{abstract}
	\noindent
	For a compact metric space $(K,\rho)$, the predual of $Lip(K,\rho)$ can be identified with the normed space $M(K)$ of finite (signed) Borel measures on $K$ equipped with the Kantorovich-Rubinstein norm, this is due to Kantorovich \cite{kantorovich1942mass}. Here we deduce atomic decomposition of $\mathcal M(K)$ by mean of some results from \cite{d2019atomic}. It is also known, under suitable assumption,  that there is a natural isometric isomorphism between $ Lip(K,\rho)$ and $(lip(K, \rho))^{**}$ \cite{Hanin}.  In this work we also show that the pair $(lip(K,\rho),Lip(K,\rho))$ can be framed in the theory of o--O type structures introduced by K. M. Perfekt.
\end{abstract}
\textup{2010} \textit{Mathematics Subject Classification}: \textup{47B44, 47G10, 26A16, 54E45} \\
 \\
\chapter{Chapter: 15 Functional Analysis}
	\section{Introduction}
	Given a metric space $(K,\rho)$ we consider the Lipschitz space $$Lip(K,\rho)=\{f:K\to\R\mid\exists L>0:\forall x,y\in K\,\,|f(x)-f(y)|\le L\rho(x,y)\}$$ (we omit $\rho$ if $K \subset \R^n$ and $\rho$ is the Euclidean distance). The quantity
		\begin{equation}
		[f]_{1}:=\sup_{x,y\in K}\frac{|f(x)-f(y)|}{\rho(x,y)}
		\end{equation}
		is a seminorm in $Lip(K,\rho)$, and it can be made into a norm by either considering functions modulo a constant or by adding a suitable quantity, such as the value of the function in a point or the supremum norm. In the following, unless specified otherwise, we will follow the latter convention:
		\begin{equation}
		\|f\|_{1}:=\max\{[f]_{1},\|f\|_\infty\}
		\end{equation}
		where $\|f\|_\infty:=\sup_{x\in K}|f(x)|$.
		\\
		The Lipschitz spaces can be generalized by considering the composition of the distance with a subadditive function $\omega$. In particular, we will consider the spaces \begin{equation}
		Lip_\alpha(K,\rho):=\left\{f:[f]_{Lip_\alpha}:=\sup_{x,y\in K}\frac{|f(x)-f(y)|}{\rho^\alpha(x,y)} <\infty \right\},
		\end{equation}
		with $0<\alpha\le 1$, corresponding to the modulus of continuity $\omega(t)=t^\alpha$. We remark that since $\rho^\alpha$ is still a metric, $Lip_\alpha(K,\rho)$ can in fact be seen as $Lip(K,\rho^\alpha)$; conversely, the space $Lip(K,\rho)$ can be seen as $Lip_\alpha(K,\rho)$ with $\alpha=1$ (i.e. $\omega(t)=t$).\\
		An important subspace of $Lip$ is the subspace
		\begin{equation}
		lip(K,\rho):=\left\{f\in Lip(K,\rho):\limsup_{\rho(x,y)\to 0}\frac{|f(x)-f(y)|}{\rho(x,y)}=0\right\}.
		\end{equation}
		In many cases, such as in the case $(\R^n,d)$, where $d$ is the Euclidean distance, the space $lip$ is trivial, but this is not the case for $lip_\alpha(K,\rho)=lip(K,\rho^\alpha)$, where $0<\alpha<1$, since an easy computation shows that it contains all $Lip_\beta(K,\rho)$ with $\alpha<\beta\le 1$ \cite{Hanin}. The importance of this subspace is that in many cases we have that
		\begin{equation}
		(lip(K,\rho))^{**}=Lip(K,\rho)
		\end{equation}
		where equality is intended as spaces being canonically isometric. De Leeuw in \cite{de1961banach} first proved that, if $0<\alpha<1$, then $(lip_\alpha([0,1]))^{**}$ is isometrically isomorphic to $Lip_\alpha([0,1])$. In (1974) Wulbert \cite{wulbert1974representations} extended de Leeuw's theorem to finite dimensional compact sets $ K$
		\begin{equation}\label{040801}
		(lip_\alpha( K))^{**}\simeq Lip_\alpha( K)
		\end{equation}
		Let us notice that de Leeuw identified the\emph{ dual} of $lip_\alpha([0,1])$ , by constructing an isometric embedding of $lip_\alpha([0,1])$ into a space $C_0(W)$ of continuous function on $W=U \cup V$
		\begin{equation}\label{040802}
		U=\{\rho\in \R: 0\leq \rho\leq 1\}
		\end{equation}
		\begin{equation}\label{040802}
		V=\left\{(\sigma, \tau)\in \R^2 : 0\leq \sigma\leq 1, 0<\tau-\sigma \leq \frac 1 2\right\}
		\end{equation}
		that are zero at infinity, equipped with the supremum norm, and then using Riesz representation Theorem. The idea was to find a normed linear space $H$ such that $H^*=Lip_\alpha$ and $(H)^c= (lip_\alpha)^*$, where $(\cdot)^c$ denotes the completion, a space of finite Borel measures on $W$, namely $H=M (W)$ with an appropriate norm. A more general version of this duality, following a similar approach, is due to Hanin \cite{Hanin} (see also \cite{bade1987amenability}) for $(K, \rho)$ a compact metric space. For this, he considered the \textit{normed} space $H=M(K)$ of all finite Borel measures $\mu$ on $ K$ equipped with the Kantorovich-Rubistein norm \cite{kantorovich1942mass,kantorovich1957functional,kantorovich1958space}. The reason why classical norm of $\mu \in M( K)$ given by total variation, was excluded by Hanin are the unclear description of the dual of $M$ with respect to such a strong norm in $BV$ and the fact that such a norm is unrelated with distance $\rho$. For example choosing the delta measures of Dirac supported at $x,y\in  K$ respectively
		$$
		|\delta_x-\delta_y|(K)=2
		$$
		while the Kantorovich-Rubinstein norm (KR-norm in short) on $M (K)=(M (K,\rho), ||\cdot||_{\rho})$ satisfies:
		$$
		||\delta_x-\delta_y||_{\rho}=\rho(x,y).
		$$
		Thus for $K$ infinite set, the space $\mathcal (K)$ is not complete.
		Let us first introduce this new norm in the space $M_0(K)$ of signed measures $\nu$ vanishing on $K$: $\nu(K)=0$. To any such a measure $\nu$ we associate the family $\Psi_{\nu}$ of all nonnegative measures $\psi\in M(K \times K)$ such that for any Borel $F\subset K$ \emph{the balance condition}
		\begin{equation}\label{040803}
		\psi(K, F) - \psi(F,K)= \nu(F)
		\end{equation}
		holds, where $\psi (F_1,F_2)$ can be interpreted as mass carried from $F_1$ to $F_2$ while $\psi \in \Psi_{\nu}$ is mass transferred on $K$ with given $\nu_{-}$ and required $\nu_+$. The norm of $\nu \in M_0( K)$ is defined by infconcolution: 
		\begin{equation}\label{040804}
		||\nu||_{\rho}^0= \inf_{\psi\in\Psi_{\nu}} \int\int_{K \times K} \rho(x,y)d\nu(x,y)
		\end{equation}
		and its value with the corresponding optimal transfer gives the solution to the mass transfer problem from $\mu^-$ to $\mu^+$ with cost $\rho$ (see \cite{gigli2011user} and reference therein); moreover for generic $\mu \in M (K)$ by
		\begin{equation}\label{040805}
		||\mu||_{\rho}=\inf_{\nu\in M_0(K)}\{||\nu||_{\rho}^0+|\mu -\nu|(K) \}
		\end{equation}
		It is worth noting that the infimum in \eqref{040804} is the same if taken over a smaller set $\bar{\Psi}_{\nu}$ of measures $\psi\in M_{+}(K \times K)$ such that
		$$
		\psi(\cdot,K)=\nu_{-}\,\,\,\psi(K, \cdot)=\nu_{+}
		$$
		where $\nu_{-}$ and $\nu_{+}$ are the negative and the positive variation of $\nu$
		and the set of measures with finite support is dense in $M(K)$ in the KR-norm.
		The main property of the space $M(K)$ endowed with the KR norm is the following duality relation (see, for instance, \cite[Theorem $0$]{Hanin}):
		\begin{equation}
		(M(K))^*=Lip(K,\rho)
		\end{equation}
		(alternatively, we also have $(M_0(K))^*=Lip(K,\rho)$, where in this case the functions in $Lip$ are considered modulo a constant).
		What Hanin did in \cite{Hanin} is to find a necessary and sufficient condition on $\rho$ to have this other duality relation:
		\begin{equation}
		(lip(K,\rho))^*=(M(K))^c,
		\end{equation}
		where $(\cdot)^c$ refers to the completion. This in particular implies the two star theorem for the pair $(lip,Lip)$.\\
		A family of distances that satisfies the assumption given by Hanin is given by $\{\tilde{\rho}^\alpha\}_{0<\alpha<1}$, where $\tilde{\rho}$ is a given distance, so in particular we obtain that $(lip_\alpha(K,\tilde{\rho}))^{**}=Lip_\alpha(K,\tilde{\rho})$, with the intermediate space being $M(K)$ endowed with the KR norm corresponding in this case to the mass transport with cost $\tilde{\rho}^\alpha$.\\
		In \cite{Perfekt} the pair $(lip_\alpha(K),Lip_\alpha(K))$ (considered modulo constants) is framed within a general type of ``o--O'' structure (see section 2) in the case $K\subset\R^n$. The o--O type structure introduced in \cite{Perfekt} is typical of several non-reflexive Banach spaces such as $BMO$ and $VMO$ (as done in the aforementioned paper), the Brezis-Bourgain-Mironescu space $B$ (introduced in \cite{bourgain2015new}) and its subspace $B_0$ (as done in \cite{d2019atomic}, see also \cite{DSSJEPE}, \cite{MR3262204}) and a particular class of Orlicz spaces (as done in \cite{angrisani2019orlicz}). This structure provides a general setting in which specific properties of Banach spaces can be shown, such as M-ideality of some subspaces (see \cite{perfekt2017m}) and a characterization of weak compact operators (see \cite{perfekt2015weak}). Concerning $lip_\alpha$ and $Lip_\alpha$ $0<\alpha<1$, in \cite{Perfekt} it is show that on compact subsets of $\R^n$ they exhibit a o--O structure, also obtaining as a consequence the M-ideality of $lip_\alpha$ in $Lip_\alpha$, which is a generalization of a result given in \cite{berninger2003lipschitz} for $K=[0,1]$. The Author also set the problem to generalize this result to a general compact metric space.\\
		The aim of this paper is to address Perfekt's question and show the $o$--$O$ structure and all of the subsequent properties for the pair $(lip(K,\rho),Lip(K,\rho))$ (with the norm $\|\cdot\|_{1}$) for a more general class of compact metric spaces. In particular we show that the necessary and sufficient conditions for the duality relation given in \cite{Hanin} are also necessary and sufficient for the $o$-$O$ structure in the class of doubling compact metric spaces. Moreover, by using the $o$-$O$ structure, we are able to show the atomic decomposition for the space $(lip_\alpha(K,\rho))^*\equiv (M(K))^c$.\\
		We remark that similar arguments work if we consider functions modulo constants, and this will give, among other things, an atomic decomposition for $(M_0(K))^c$.
	\section{$o$--$O$ structure for Banach spaces}
	\subsection{The $o$--$O$ structure}
Let us recall the definition of o--O structure for pair of Banach spaces.
\begin{defn}
	We say that a pair of Banach spaces $(E_0,E)$, where $E_0$ is a subspace of $E$, form a o--O structure if there exist
	\begin{itemize}
		\item a reflexive separable Banach space $X$;
		\item a Banach space $Y$;
		\item a family $\cL$ of bounded linear operators from $X$ to $Y$;
		\item a topology $\tau$ on $\cL$ that is $\sigma$-compact locally compact Hausdorff topology and such that the maps $T_x:\cL \to Y$ given by $T_x L=Lx$ for any $x \in X$ are continuous;
	\end{itemize}
	such that $E$ is given by
	\begin{equation*}
	E=\left\{x \in X: \ \sup_{L \in \cL}\Norm{Lx}{Y}<+\infty\right\},
	\end{equation*}
	it holds $\Norm{x}{E}=\sup_{L \in \cL}\Norm{Lx}{Y}$ and $E_0$ is given by
	\begin{equation*}
	E_0=\left\{x \in E: \ \limsup_{L \to \infty}\Norm{Lx}{Y}=0\right\}
	\end{equation*}
	where $L \to \infty$ is intended in the Alexandrov one point compactification of $(\cL, \tau)$.\\
	Moreover, we say that the o--O pair $(E_0,E)$ satisfies \textbf{assumption AP} if and only if for any $x \in E$ there exists a sequence $\{x_n\}_{n \in \N} \subseteq E_0$ such that $x_n \rightharpoonup x$ in $X$ and $\sup_{n \in \N}\Norm{x_n}{E}\le \Norm{x}{E}$.
\end{defn}
\begin{defn}[\cite{harmand2006m}]
	Let $X$ be a Banach space. A linear projection $P$ is called an $L$-projection if for any $x \in X$
	\begin{equation*}
	\Norm{x}{X}=\Norm{Px}{X}+\Norm{x-Px}{X}.
	\end{equation*}
	A closed subspace $J \subseteq X$ is called an $L$-summand if it is the range of an $L$-projection.\\
	Given a subspace $J \subseteq X$, the subspace \mbox{$J^\perp=\{x^* \in X^*: \ x^*(y)=0 \ \forall y \in J\}$} of $X^*$ is called annihilator of $J$.\\
	A closed subspace $J \subseteq X$ is called $M$-ideal in $X$ if its annihilator $J^\perp$ is an $L$-summand of $X^*$.
\end{defn}
K. M. Perfekt proves the following properties of this structure.
	\begin{thm}
	Under \textsc{\textbf{Assumption AP}}, $E_0^{**}$ is isometrically isomorphic to $E$.\\ It can also be proved $E_0^*$ is the \textbf{strongly unique predual} of $E$, that is any other Banach space that has dual space isometric to $E$ is itself isometric to $E_0$.\\
	Also, $E_0$ is an $M$-ideal in $E$ and the following distance formula holds:
	$$\text{dist}(x,E_0)_{E}={\limsup\limits_{ L \to \infty}} \|Lx\|_{Y}.$$
\end{thm}

\subsection{Abstract atomic decomposition}
Let us show the following abstract result on $o$--$O$ structures.
	\begin{thm}
		Let $(E_0,E)$ be an $o$-$O$ pair such that the space of the operators $(\cL,\tau)$ is separable.
		Then there exists a constant $C\in (0,1)$ such that for any $\Phi \in E_*$ there exist two sequences $(g_n)_{n \in \N}\subset E_*$ and $(\lambda_n)_{n \in \N}\in \ell^1(\R_+)$ such that $\Norm{g_n}{E_*}= 1$, $$\Phi=\sum_{n \in \N}\lambda_n g_n$$ and $$C\sum_{n=1}^{+\infty}\lambda_n \le \Norm{\Phi}{E_*}\le \sum_{n=1}^{+\infty}\lambda_n$$.
	\end{thm}
	\begin{proof}
		Since we have that the map $T_x: L \in \cL \mapsto Lx \in Y$ is continuous, we can also observe that
		\begin{equation*}
		E=\left\{x \in X: \ \sup_{L \in D}\Norm{Lx}{Y}<+\infty \right\}.
		\end{equation*}
		From \cite[Theorem $3$]{d2019atomic} we know that for $\Phi \in E_*$ there exists a sequence $(y^*_n)_{n \in \N}\in \ell^1(Y^*)$ such that
		\begin{equation*}
		\Phi=\sum_{n=1}^{+\infty}L^*_ny^*_n
		\end{equation*}
		for some sequence $\{L_n\}_{n \ge 0} \subset \cL$ independent from $\Phi$, where $L^*_n \in \cL(Y^*,E_*)$ is the adjoint operator of $L_n$. Now let us recall that $\Norm{L_n}{\cL(E,Y)}=\Norm{L_n^*}{\cL(Y^*,E_*)}$. By definition of $E$, we can use the Banach-Steinhaus theorem to assure that there exists a constant $K_1$ such that
		\begin{equation*}
		\Norm{L_n^*}{\cL(Y^*,E_*)}=\Norm{L_n}{\cL(E,Y)} \le K_1.
		\end{equation*}
		Let us then define
		\begin{equation*}
		g_n=\frac{L_n^*y_n^*}{\Norm{L_n^*y_n^*}{E_*}}
		\end{equation*}
		($g_n=0$ if $L_n^*y_n^*=0$) and $\lambda_n=\Norm{L_n^*y_n^*}{E_*}$ to obtain
		\begin{equation*}
		\Phi=\sum_{n =1}^{+\infty}\lambda_n g_n.
		\end{equation*}
		Now let us observe that, since $\Norm{g_n}{E_*}=1$,
		\begin{equation*}
		\Norm{\Phi}{E_*}\le \sum_{n =1}^{+\infty}\lambda_n.
		\end{equation*}
		Let us also observe that
		\begin{equation*}
		\Norm{L_n^*y_n^*}{E_*}\le \Norm{L_n^*}{\cL(Y^*,E_*)}\Norm{y_n^*}{Y^*}\le K_1 \Norm{y_n^*}{Y^*}
		\end{equation*}
		so that
		\begin{equation*}
		\Norm{y_n^*}{Y^*}\ge \frac{1}{K_1}\lambda_n.
		\end{equation*}
		Since the predual is strongly unique, by using the isometry in \cite[Theorem $3$]{d2019atomic} we have that there exists a constant $K_2$ (independent from $\Phi$) such that
		\begin{equation*}
		\Norm{\Phi}{E_*}\ge K_2 \sum_{n=1}^{+\infty}\Norm{y^*_n}{Y^*} \ge \frac{K_2}{K_1}\sum_{n=1}^{+\infty} \lambda_n.
		\end{equation*}
		Pose $C=\frac{K_2}{K_1}$ to conclude the proof.
	\end{proof}
\section{Besov and fractional Haj\l{}asz-Sobolev spaces on compact metric measure spaces}
\subsection{Doubling metric spaces and doubling measures}
In this subsection we recall the notions of doubling measure space and doubling measure, see for example \cite{ambrosiotilli} and the references therein.
\begin{defn}
	We say that a metric space $(K,\rho)$ has the \emph{doubling condition} if there exists a positive integer $C$ such that any ball $B$ can be covered by at most $C$ balls having half the radius.\\
	A Borel measure $\mu$ on a metric space $(K,\rho)$ is said to have the \emph{doubling condition} if
	\begin{itemize}
	\item[(i)] there exist two balls $B_1, B_2$ such that $\mu(B_1)>0$ and $\mu(B_2)<+\infty$;
	\item[(ii)] there exists a constant $C>0$ such that
	\begin{equation}\label{doubcon}
	\mu(B_{2r}(x))\le C\mu(B_r(x))
	\end{equation}
	for all $x\in K$ and all $r>0$. The space $(K, \rho, \mu)$ is said to be a doubling metric measure space. A measure $\mu$ that satisfies (i) is said to be non-degenerate. Condition (ii) is called doubling condition and the constant $C$ is called doubling constant.
	\end{itemize}
\end{defn}
First of all, let us observe that the choice of the constant $2$ is arbitrary. Indeed, if we fix a constant $c>1$, one can show that a non-degenerate measure $\mu$ is a doubling measure on $(K, \rho)$ if and only if there exists a constant $C_c>0$ such that
\begin{equation}\label{doubgen}
\mu(B_{cr}(x))\le C_c\mu(B_r(x))
\end{equation}
for all $x \in K$ and all $r>0$.\\
This property implies that any doubling measure $\mu$ is fully supported. Indeed, let us consider a generic $x \in K$ and $r>0$. Let us suppose by contradiction that $\mu(B_r(x))=0$. Then $B_1$ is not contained in $B_r(x)$. However, there exists a $c>1$ such that $B_1 \subseteq B_{cr}(x)$, but by doubling condition \eqref{doubgen} we have $\mu(B_{cr}(x))=0$, concluding the proof.\\
Moreover, if $(K, \rho, \mu)$ is a compact doubling metric measure space, then $\mu(K)<+\infty$. Indeed, if we consider $x \in K$ and $r>0$ such that $B_2=B_r(x)$, since $K$ is compact, there exists a constant $c\ge 1$ such that $B_{cr}(x)=K$, concluding that $\mu(B_{cr}(x))<+\infty$ by doubling condition \eqref{doubgen}.\\
It is easy to see that if a measure is doubling then the underlying metric space must be doubling \cite{coifman2006analyse}, while the converse is not true in general. However in \cite{luukkainen1998every} it is shown that every complete (and in particular compact) doubling metric space can be given a doubling measure.\\
We will also work with compact metric spaces of the form $(K, \rho^\alpha)$ for some metric $\rho$. First of all, let us show the following easy Lemma.
\begin{lem}
	Fix $\alpha \in (0,1)$. If $(K,\rho)$ is a compact doubling metric space then $(K, \rho^\alpha)$ is a doubling metric space. Moreover, any doubling measure $\mu$ on $(K, \rho)$ is also doubling on $(K, \rho^\alpha)$.
\end{lem}
\begin{proof}
	First of all, since $(K, \rho)$ is a doubling metric space, then there exists a doubling measure $\mu$ on $(K, \rho)$. Using the doubling condition in the form \eqref{doubgen} for $c=2^\alpha$ and setting $C_{2^\alpha}=C_\alpha$ we have
	\begin{equation*}
	\mu(B_{(2r)^\alpha}(x))\le C_{\alpha} \mu(B_{r^\alpha}(x))
	\end{equation*}
	hence $\mu$ is a doubling measure on $(K, \rho^\alpha)$. Finally, since we have a doubling measure on $(K,\rho^\alpha)$, $(K, \rho^\alpha)$ is a doubling metric space.
\end{proof}
Moreover, let us observe, as a consequence of \cite[Lemma $4.7$]{hajlasz2003sobolev}, the following Lemma.
\begin{lem}\label{lemdoub}
	Let $(K, \rho)$ be a compact doubling metric space and $\mu$ a doubling measure on it. Then there exist two constants $C,Q>0$ such that
	\begin{equation*}
	\mu(B_r(x))\ge C r^Q.
	\end{equation*}
\end{lem}
\subsection{Besov spaces on metric measure spaces}
In the following we will need the notion of Besov spaces on metric measure spaces. From now on, in this section, let us fix a doubling compact metric space $(K,\rho)$ and a doubling measure $\mu$ on $K$.
\begin{defn}[\cite{gogatishvili2010interpolation}]
	The \textbf{Besov} space of parameters $s\in (0,1)$ and $p,q \in [1,\infty)$ on $(K,\rho,\mu)$ is the space
	\begin{multline*}
	\cB^s_{p,q}(K,\rho,\mu)=\Big\{f:K\to \R: f \in L^p(K,\mu) \text{ and } \\ [f]_{\cB^s_{p,q}}:=\left[\int\limits_0^{+\infty} \frac{dr}{r} \left[\int_{K} \fint_{B_r(x)} \frac{|f(x)-f(y)|^{p}}{r^{sp}}\,d\mu(y)\,d\mu(x)\right]^{q/p}\right]^{1/q}<+\infty \Big\}.\end{multline*}
	It is a Banach space when endowed with the norm
	$$\|f\|_{\cB^s_{p,q}}=\|f\|_{L^p}+[f]_{\cB^s_{p,q}}.$$
\end{defn}
The space $\cB_{p,q}^s$ is obviously separable as it embeds continuously as a subspace of $L^p(K,\rho)$ which is separable because $L^p$ spaces on separable metric measure spaces are separable.
\\
With the additional assumption that $p=q>2$, a \textbf{Clarkson type inequality} (see \cite{Grigoryan}) can be proved, in the sense that follows: as it is more convenient for technical reasons, introduce the equivalent norm $\|f\|'_{\cB^s_{p,q}}=(\|f\|_{L^p}^p+([f]^{s}_{p,q})^p)^{1/p}$ and then, in a way similar to how it is done in $L^p$ spaces, prove:
$$\left\|\frac{f+g}{2}\right\|^p_{\Lambda^s_{p,p}}+\left\|\frac{f-g}{2}\right\|^p_{\cB^s_{p,p}}\le \frac{1}{2}\left[ \left\|f\right\|^p_{\cB^s_{p,p}}+\left\|g\right\|^p_{\cB^s_{p,p}}\right]$$
This allows to prove that $\cB^s_{p,p}$ is \textbf{uniformly convex}, that is for every $\varepsilon>0$ there exists $\delta>0$ such that:
$$\|f\|_{\cB^s_{p,p}},\|g\|_{\cB^s_{p,p}}=1 \text{ and } \left\|\frac{f+g}{2}\right\|_{\cB^s_{p,p}}>1-\delta \Rightarrow \|f-g\|_{\cB^s_{p,p}}<\varepsilon.$$
Recalling that a theorem by Milman and Pettis shows that every uniformly convex Banach space is reflexive (see for instance \cite{brezis2010functional}), we know that $\cB^s_{p,q}$ is a reflexive and separable Banach space.\\
In \cite{gogatishvili2010interpolation} it has been shown that the seminorm $[f]_{\cB_{p,p}^s}$ is equivalent to the semi-norm
\begin{equation*}
[f]_{\cB_p^s}=\int_{K}\int_{K}\frac{|f(x)-f(y)|^p}{\rho(x,y)^{\alpha p}\mu(B_{\rho(x,y)}(x)}d\mu(x)d\mu(y).
\end{equation*}
In particular, if $K=\R^n$, $\rho$ is the Euclidean distance and $\mu$ is the Lebesgue measure, then $[\cdot]_{\cB_p^s}$ is the semi-norm characterizing the fractional Sobolev space $W^{s,p}(\R^n)$, as used in \cite{Perfekt}.
\subsection{Fractional Haj\l{}asz-Sobolev spaces on metric measure spaces}
Let us give a definition of another interesting functional space on metric measure spaces.
\begin{defn}[\cite{karak2019measure}]
	Let $f \in L^p(K,\mu)$. Then we say that $g$ is a Haj\l{}asz $s$-gradient if
	\begin{equation*}
	|f(x)-f(y)|\le \rho^s(x,y)|g(x)+g(y)|.
	\end{equation*}
	We denote with $\cD_s(f)$ the set of all Haj\l{}asz gradents of $f$ and with $\cD_s^p(f)$ the set of the Haj\l{}asz gradients of $f$ that are in $L^p(K,\mu)$. We say that $f$ belongs to the fractional Haj\l{}asz-Sobolev space $\cH^{s,p}(K,\rho,\mu)$ if $\cD_s^p(f) \not = \emptyset$. In particular $\cH^{s,p}(K,\rho,\mu)$ is a Banach space when endowed with the norm
	\begin{equation*}
	\Norm{f}{\cH^{s,p}}:=\Norm{f}{L^p}+[f]_{\cH^{s,p}}
	\end{equation*}
	where
	\begin{equation*}
	[f]_{\cH^{s,p}}:=\inf_{g \in \cD_s^p(f)}\Norm{g}{L^p}
	\end{equation*}
\end{defn}
Let us first observe that the fractional Haj\l{}asz-Sobolev space $\cH^{s,p}(K,\rho,\mu)$ coincides with the Haj\l{}asz-Sobolev space $\cH^{1,p}(K,\rho^s,\mu)$. Hence, in particular, a Morrey-type embedding theorem can be shown, as a direct consequence of \cite[Theorem $8.7$]{hajlasz2003sobolev}.
\begin{thm}
	Let $(K,\rho,\mu)$ be a compact metric measure space such that $\mu(B_r(x))\ge Cr^Q$ for some constants $C>0, Q\ge0$ and for all $x\in K, r>0$, and suppose $p>\frac{Q}{s}$. Then there exists a constant $C$ such that for any function $u \in \cH^{s,p}(K,\rho,\mu)$ and any $g \in \cD_s^p(u)$ it holds
	\begin{equation*}
	|u(x)-u(y)| \le C d(x,y)^{s-\frac{Q}{p}}\Norm{g}{L^p}, \ \forall x,y \in K.
	\end{equation*}
\end{thm}
Other embedding theorems can be shown, also for more general spaces (see, for instance, \cite{karak2019measure}).\\
From this punctual estimate we deduce that if $u \in \cH^{s,p}(K,\rho,\mu)$ for $p>\frac{Q}{s}$ then $u$ is $\left(s-\frac{Q}{p}\right)$-H\"older continuous. In Lemma \ref{lemdoub} we have shown that doubling measures satisfy the previous condition for some $Q\ge 0$. From now on let $Q$ be such a constant. We have
\begin{prop}
	Let $(K,\rho,\mu)$ be a doubling compact metric measure space and $p>\frac{Q}{s}$. Then $\cH^{s,p}(K,\rho,\mu)$ embeds with continuity in $L^\infty$.
\end{prop}
\begin{proof}
	Let us denote with $D=\diam(K)$. Fix $u \in \cH^{s,p}(K,\rho,\mu)$, $g \in \cD_s^p(u)$ and $x,y \in K$. Observe that, by using the previous punctual estimate,
	\begin{align*}
	|u(x)|&\le |u(y)|+|u(x)-u(y)|\\
	&\le |u(y)|+Cd(x,y)^{s-\frac{Q}{p}}\Norm{g}{L^p}\\
	&\le |u(y)|+CD^{s-\frac{Q}{p}}\Norm{g}{L^p}.
	\end{align*}
	By using the $p$-homogeneity and convexity of the function $t \mapsto t^p$ for $t>0$ we have
	\begin{equation*}
	|u(x)|^p \le C_1(|u(y)|^p+C^pD^{ps-Q}\Norm{g}{L^p}^p)
	\end{equation*}
	and then integrating in $d\mu(y)$, setting $M=\mu(K)$, we have
	\begin{equation*}
	|u(x)|^p \le \frac{C_1}{M}(\Norm{u}{L^p}^p+MC^pD^{ps-Q}\Norm{g}{L^p}^p).
	\end{equation*}
	Now, by using the fact that there exists a constant $C_p$ such that $(a^p+b^p)^{\frac{1}{p}}\le C_p(a+b)$ for any $a,b>0$, we have
	\begin{equation*}
	|u(x)| \le C_2(\Norm{u}{L^p}+\Norm{g}{L^p}).
	\end{equation*}
	Now let us take the infimum over $\cD_s^p(u)$ to achieve
	\begin{equation*}
	|u(x)|\le C_2 \Norm{u}{\cH^{s,p}},
	\end{equation*}
	and then, taking the maximum on $K$, we have
	\begin{equation*}
	\Norm{u}{L^\infty}\le C_2 \Norm{u}{\cH^{s,p}}.
	\end{equation*}
\end{proof}
Concerning the relation between $\cB_{p,p}^s$ and $\cH^{s,p}$, one can show the following embedding theorem as a direct consequence of \cite[Lemma $6.1$]{gogatishvili2010interpolation}.
\begin{thm}
	Let $(K,\rho,\mu)$ be a doubling compact metric measure space. Then $\cB_{p,p}^s(K,\rho,\mu)$ embeds with continuity in $\cH^{s,p}(K,\rho,\mu)$.
\end{thm}
Actually, the statement of \cite[Lemma $6.1$]{gogatishvili2010interpolation} only refers to the inclusion of $\cB_{p,p}^s(K,\rho,\mu)$ in $\cH^{s,p}(K,\rho,\mu)$. However, in the proof, it is shown that there exists a constant $C>0$, depending only on the doubling constant of $\mu$, such that for any $u \in \cB_{p,p}^s(K,\rho,\mu)$ there exists a $g \in \cD_s^p(u)$ such that $[u]_{\cH^{s,p}}\le \Norm{g}{L^p}\le C [u]_{\cB_{p,p}^s}$. Summing on both sides $\Norm{u}{L^p}$ one has the continuous embedding.\\
As a Corollary of the previous two embedding theorems we have the following
\begin{cor}\label{corHj}
	Let $(K,\rho,\mu)$ be a doubling compact metric measure space and $p>\frac{Q}{s}$. Then $\cB^s_{p,p}(K,\rho,\mu)$ embeds continuously in $L^\infty$.
\end{cor}
\section{Lipschitz spaces on compact metric measure spaces}
\subsection{Notations for Lipschitz spaces}
Let us now introduce some notations and basic definitions on Lipschitz spaces defined on compact metric spaces. Let $(K,\rho)$ be a compact metric space. Without loss of generality, we can suppose that $\mu$ is a probability measure. Moreover, let us call $D=\text{diam}(K)$.\\
We will denote:
\begin{itemize}
	\item $\text{Lip}(K,\rho)=\left\{f:K\to \R: [f]_{1}:=\sup\limits_{x\not = y \in K} \frac{|f(x)-f(y)|}{\rho(x,y)}<+\infty \right\}$
	\item $\text{lip}(K,\rho)=\left\{f \in \text{Lip}(K,\rho): \lim\limits_{\rho(x,y)\to 0} \frac{|f(x)-f(y)|}{\rho(x,y)}=0 \right\}$.
\end{itemize}
The first of those two spaces is a Banach space with the norm given by
\begin{equation}
\|f\|_{1}:=\max\{[f]_{1},\|f\|_\infty\}
\end{equation}
where $\|f\|_\infty=\max_{x \in K}|f(x)|$.\\
Let us also introduce the H\"older spaces for $\alpha \in (0,1)$:
\begin{itemize}
	\item $\text{Lip}_\alpha(K,\rho)=\left\{f:K\to \R: [f]_{\alpha}:=\sup\limits_{x\not = y \in K} \frac{|f(x)-f(y)|}{\rho(x,y)^{\alpha}}<+\infty \right\}$
	\item $\text{lip}_\alpha(K,\rho)=\left\{f \in \text{Lip}_\alpha(K,\rho): \lim\limits_{\rho(x,y)\to 0} \frac{|f(x)-f(y)|}{\rho(x,y)^\alpha}=0 \right\}$.
\end{itemize}
The first of those two spaces is a Banach space with the norm given by
\begin{equation}
\|f\|_{\alpha}:=\max\{[f]_\alpha,\|f\|_\infty\}.
\end{equation}
While $lip(K,\rho)$ could coincide with the space of the constant functions, it holds $Lip(K,\rho) \subset lip_\alpha(K,\rho)$, hence $lip_\alpha(K,\rho)$ cannot be trivial.\\
Actually the spaces $Lip_\alpha(K,\rho)$ are particular cases of Lipschitz spaces. Indeed we have $Lip_\alpha(K,\rho)=Lip(K,\rho^\alpha)$. In the following we will set $Lip_1(K,\rho):=Lip(K,\rho)$.
\subsection{Embedding of $Lip_\alpha$ in $\cB^s_{p,p}$}
Let us now focus on compact metric measure spaces $(K,\rho,\mu)$. The following result exhibit the relation between $Lip(K,\rho)$ and $\cB^s_{p,p}(K,\rho,\mu)$ ad actually holds whenever $\mu$ is a fully supported finite measure on $(K,\rho)$. However, to fix the ideas, let us consider $(K,\rho,\mu)$ as a doubling compact metric measure space.
\begin{prop}
	Consider $\alpha \in (0,1]$. Then the space $Lip_\alpha(K,\rho)$ continuously embeds in $\cB^s_{p,p}$ for $s \in (0,\alpha)$ and $p \in [1,+\infty)$.
\end{prop}
\begin{proof}
	To prove that $\text{Lip}_\alpha(K,\rho)$ embeds continuously in $X$, it is necessary to assume $s<\alpha$.\\
	Indeed, if so, assume that $C=\|f\|_1$. Then $\|f\|_{L^p}\le\|f\|_{L^\infty}\le C$ and \begin{multline*}\left[\int_0^{+\infty} \frac{dr}{r} \int_K\fint_{B_r(x)} \frac{|f(x)-f(y)|^p}{r^{sp}}\,d\mu(y)\,d\mu(x)\right]^{1/p}  \le \\ \le \left[\int_0^{D} \frac{1}{r} C^p r^{(\alpha-s)p}dr+\int_D^{+\infty} 2C^p \,\frac{dr}{r^{1+sp}} \right]^{1/p} \le kC\end{multline*}
	where in the first integral the idea was using the fact that $C$ bounds the Lipschitz constant, while in the second one the fact that $C$ bounds the $L^\infty$ norm of $f$ was used.\\
\end{proof}
\subsection{Approximation property for Lipschitz spaces}
Here we want to show an approximation property in $Lip(K,\rho)$, i. e. for any function $f \in  Lip(K,\rho)$ we want to find a sequence of functions in a suitable subspace that pointwise converges towards $f$. To do this, we need to introduce the following \emph{separation} assumption, that is shown in \cite{Hanin} to be equivalent to the isometry between $(M(K))^c$ and $(lip(K,\rho))^*$ see also \cite{berninger2003lipschitz}.
\begin{assumptionH}
For any $f\in Lip(K,\rho)$, $A$ a finite subset of $K$ and $C>1$ real constant, there exists a function $g\in lip(K,\rho)$ such that $g|_A\equiv f|_A$ and $\|g\|\le C\|f\|$.
\end{assumptionH}
Hence we have the following Theorem.
\begin{thm}
	Let us suppose \textsc{\textbf{Assumption H}} holds. Let $f\in Lip(K,\rho)$. There is a sequence $\{f_n\}_{n\in\N}\subset lip(K,\rho)$ pointwise converging to $f$ and such that $\sup_{n\in\N}\|f_n\|_1\le\|f\|_1$.
\end{thm}
\begin{proof}
	Since $K$ is totally bounded, it can be covered by a finite number of balls of radius $1$, so let's call $A_0$ the set of centers of these balls. Suppose now that we have defined the set $A_n$ and consider the set $K_{n+1}:=K\backslash\bigcup_{x\in A_n}B_{2^{-n-1}}(x)$. Since $K_{n+1}$ is a compact and thus totally bounded subset of $K$, it can be covered by balls of radius $2^{-n-1}$, so if we denote by $B_{n+1}$ the corresponding set of centers, we can take $A_{n+1}:=A_n\cup B_{n+1}$. This ensures that every point of $K$ has distance less that $2^{-n}$ from the points in $A_n$. We also take $C_n:=1+\frac{1}{n+1}$.\\
	Let $g_n$ be the function from \textsc{\textbf{Assumption H}} obtained by considering $A=A_n$ and $C=C_n$ and define $f_n:=\frac{g_n}{C_n} \in lip(K,\rho)$. We have that $\|f_n\|_\alpha\le\|f\|_\alpha$, so the only thing that's left to show is the pointwise convergence, which implies weak convergence in $X$. We notice that, by definition of $f_n$, it is enough to show that $g_n\to f$ pointwise. If we define $A_\infty:=\bigcup_{n\in\N} A_n$ we see that $A_\infty$ is dense and for all $x\in A_\infty$ the sequence $g_n(x)$ eventually becomes constantly equal to $f(x)$. By using the Lipschitz property we can easily extend the pointwise convergence to the whole $K$.\\
\end{proof}
Let us also observe that in the case of $Lip_\alpha(K,\rho)$, \textsc{\textbf{Assumption H}} can be actually improved \cite[Proposition 3]{Hanin}.
\begin{prop}
	Let $(K,\rho)$ be a compact metric space, $\alpha\in(0,1)$, $f\in Lip_\alpha(K)$, $A$ a finite subset of $K$ and $C>1$ a real constant. Then there exists a function $g\in Lip_1(K)$ such that $g|_A\equiv f|_A$ and $\|g\|_\alpha\le C\|f\|_\alpha$.\\
\end{prop}

\section{The $o$--$O$ structure of $(lip(K,\rho),Lip(K,\rho))$}
Now we are ready to show the main result of the paper.
\begin{thm}
	Let $(K,\rho,\mu)$ be a doubling compact metric measure space. Then the pair $(lip(K,\rho),Lip(K,\rho))$ exhibit a $o$--$O$ structure if and only if \textsc{\textbf{Assumption H}} holds. Supposing that \textsc{\textbf{Assumption H}} holds, as a consequence we have the following properties:
	\begin{itemize}
		\item $(lip(K,\rho))^{**}\simeq Lip(K,\rho)$ isometrically;
		\item for $f\in Lip(K,\rho)$ the following distance formula holds:
		\begin{equation}
		dist_{Lip(K,\rho)}(f,lip(K,\rho))=\limsup_{\rho(x,y)\to 0}\frac{|f(x)-f(y)|}{\rho(x,y)};
		\end{equation}
		\item $lip(K,\rho)$ is an $M$-ideal in $Lip(K,\rho)$, that is
		\begin{equation}
		(Lip(K,\rho))^*\simeq (lip(K,\rho))^*\oplus_1 (lip(K,\rho))^\perp,
		\end{equation}
		\item $(lip(K,\rho))^*$ is the strongly unique predual of $Lip(K,\rho)$.
	\end{itemize}
\end{thm}
\begin{proof}
	First of all we need a reflexive and separable Banach space $X$ in which we can embed $\text{Lip}(K,\rho)$. Thus fix $s \in (0,1)$ and $p>\frac{Q}{s}$ and consider $X=\overline{Lip(K,\rho)}^{\cB^s_{p,p}}$, since we have shown that $\cB^s_{p,p}$ is reflexive, separable and $Lip(K,\rho)$ continuously embeds in it, so the closure of $Lip$ in $\cB_{p,p}^s$ also has these properties and will be our $X$.\\
	As Banach space $Y$ let us choose $\R\times\R$, endowed with the $L^\infty$ norm, i.e. $\|(x,y)\|_{\R\times\R}=\max\{|x|,|y|\}$.\\
	Our family of operators will be the following:
	$$\mathcal{L}=\left\{L_{x,y,z}:f \in X\mapsto \left(\frac{f(x)-f(y)}{\rho(x,y)},\frac{\rho(x,y)}{D} f(z)\right) \in \R\times\R,x,y,z \in K, x\not=y\right\}.$$
	It is clear that these operators are linear.\\
	If we set $V:=K^2\backslash \text{Diag}(K^2)$, we can give $\mathcal{L}$ the product topology of $V\times K$, where on $V$ we have the trace topology induced by the topology on $K^2$. In the following we will identify $\cL$ with $W:=V \times K$.\\
	Since $K$ is a compact metric space, it is $\sigma$-compact, locally compact, Hausdorff and separable and so is also $V$. These properties easily transfer to $\mathcal{L}$, being it a product space. In particular an exhaustive sequence $K_n$ of compact subsets of $\mathcal{L}$ is given by
	\begin{equation*}
	K_n=\left\{(x,y)\in K^2: \ \rho(x,y) \ge \frac{1}{n}\right\}
	\end{equation*}
	hence taking the limit as $L \to \infty$ is equivalent to taking the limit as $\rho(x,y)\to 0$.
	\medskip
	Now we need to show the continuity of the maps $T_f: L\in \mathcal{L} \mapsto L(f) \in \R \times \R$ for $f\in X$. We notice that it is enough to prove this for $f\in Lip(K,\rho)$, since we can use a diagonal argument, combined with the boundedness of the operators themselves, to extend this to the whole $X$.\\
	This is easy because $(x_n,y_n,z_n)\to (x,y,z)$ as $n$ goes to infinity implies $\rho(x_n,y_n)\to \rho(x,y)$ and $\rho(z_n,z)\to 0$, so using the continuity of $f$ and $\rho$ we easily obtain $$\max\left\{\left|\frac{f(x_n)-f(y_n)}{\rho(x_n,y_n)} - \frac{f(x)-f(y)}{\rho^\alpha(x,y)}\right|,|\frac{\rho(x_n,y_n)}{D} f(z_n)-\frac{\rho(x,y)}{D} f(z)|\right\}\to 0$$
	proving that $T_f$ is continuous for any $f \in \text{Lip}(K,\rho)$.\\
	It is easy to observe that
		\begin{multline*}
		\sup_{(x,y,z)\in W}\Norm{L_{x,y,z}f}{\R\times \R}=\sup_{(x,y,z)\in W}\max\left\{\left|\frac{f(x)-f(y)}{\rho(x,y)}\right|,\frac{\rho(x,y)}{D}|f(z)|\right\}\\=\max\{[f]_1,\Norm{f}{\infty}\}=\Norm{f}{1},
		\end{multline*}
while the $o$-structure for $lip(K,\rho)$ follows from the inequality
\begin{equation*}
\frac{|f(x)-f(y)|}{\rho(x,y)}\le \Norm{L_{x,y,z}f}{\R \times \R} \le \frac{|f(x)-f(y)|}{\rho(x,y)}+\frac{\rho(x,y)}{D}\Norm{f}{L^\infty}.
\end{equation*}
Concerning the continuity of $L_{x,y,z}$, let us recall, from Corollary \ref{corHj}, that there exists a constant $C$ such that $\Norm{f}{\cB_{p,p}^s} \ge \Norm{f}{L^\infty}$. Hence we have
\begin{equation*}
\frac{|f(x)-f(y)|}{\rho(x,y)\Norm{f}{\cB_{p,p}^s}} \le \frac{2}{C\rho(x,y)}
\end{equation*}
while
\begin{equation*}
\frac{\rho(x,y)|f(z)|}{D\Norm{f}{\cB_{p,p}^s}} \le \frac{\rho(x,y)}{CD},
\end{equation*}
thus $L_{x,y,z}:X \to \R \times \R$ is a bounded linear operator.\\
Finally let us observe that we have shown that, supposed that \textsc{\textbf{Assumption H}} holds, for any $f \in Lip(K,\rho)$ there exists a sequence $\{f_n\}_{n \in \N}\subset lip(K,\rho)$ such that $f_n \to f$ point-wise and $\sup_{n \in \N}\Norm{f_n}{1}\le \Norm{f}{1}$, hence, by Banach-Alaoglu theorem, we can extract a subsequence of $f_n$ that weakly converges to $f$ in $X$,  concluding the proof of one implication.\\
Concerning the other implication, let us suppose that the $o$--$O$ structure holds. Then we know that $(lip(K,\rho))^{**}\simeq Lip(K,\rho)$ isometrically. However, in \cite{Hanin} it is shown that such isometry is equivalent to \textsc{\textbf{Assumption H}}, concluding the proof.
\end{proof}
\begin{rmk}
	Let us remark that, since we have shown that \textsc{\textbf{Assumption H}} always holds for $Lip_\alpha(K,\rho)$ whenever $\alpha \in (0,1)$, we have that the pair $(lip_\alpha(K,\rho),Lip_\alpha(K,\rho))$ always exhibits the $o$--$O$ structure.
\end{rmk}
\section{The atomic decomposition of $M(K)$}
Now we are ready to give an atomic decomposition of the space $M(K)$ endowed with the Kantorovich-Rubinstein norm.
\begin{thm}
	Fix $\alpha \in (0,1)$. Let $\mu \in M(K)$. Then there exist a sequence of atomic measures $(\mu_n)_{n \in \N}\subset M(K)$ with $card(supp(\mu_n))\le 3$ and a sequence $(\gamma_n)_{n \in \N}\in \ell^1(\R)$ with $\gamma_n \ge 0$ such that
	\begin{equation*}
	\mu=\sum_{n=1}^{+\infty}\gamma_n \mu_n
	\end{equation*}
	where the convergence is intended in the Kantorovich-Rubinstein norm with respect to $\rho^\alpha$. Moreover there is $C>0$ such that
	\begin{equation}
	C\sum_{n=1}^{+\infty}\gamma_n \le \Norm{\mu}{M(K)}\le \sum_{n=1}^{+\infty}\gamma_n
	\end{equation}
\end{thm}
\begin{proof}
	Since we have shown that the pair $(lip_\alpha(K,\rho),Lip_\alpha(K,\rho))$ admits a $o$--$O$ structure whenever $\alpha \in (0,1)$, we know that $(lip_\alpha(K,\rho))^*$ is the strongly unique predual of $Lip_\alpha(K,\rho)$ and $(lip_\alpha(K,\rho))^*\simeq (M(K))^c$. Moreover, let us observe that the topology on $\cL$ is the one induced by the natural topology on $V \times K \times (0,1]$, hence it is separable. Thus we know that there exist a sequence $(\gamma_n)_{n \in \N} \in \ell^1(\R)$ with $\gamma_n \ge 0$ and a sequence $(\mu_n)_{n \in \N} \subset (M(K))^c$ such that
	\begin{equation*}
	\mu=\sum_{n=1}^{+\infty}\gamma_n \mu_n
	\end{equation*}
	with $\Norm{\mu_n}{M(K),\rho^\alpha}=1$, where with $\Norm{\cdot}{M(K),\rho^\alpha}$ we denote the Kantorovich-Rubinstein norm.\\
	Now recall that $\mu_n=\frac{L_n^*g_n^*}{\gamma_n}$ by definition for some operators $L_n \in \cL$ and for some $g_n^* \in Y^*=\R^2$. Let us observe that for any $g \in \R^2$ and $f \in Lip_\alpha(K,\rho)$ we have
	\begin{equation*}
	\langle L_nf, g \rangle= \frac{f(x_n)-f(y_n)}{\rho(x_n,y_n)^\alpha}g_1+\frac{\rho(x_n,y_n)^{\alpha}}{D^\alpha} f(z_n)g_2
	\end{equation*}
	while
	\begin{equation*}
	\langle f, L_n^*g \rangle= \int_{K}f dL_n^*g.
	\end{equation*}
	From these two relations we have for any $g \in \R^2$
	\begin{equation*}
	L_n^*g=\frac{\delta_{x_n}-\delta_{y_n}}{\rho(x_n,y_n)^\alpha}g_1+\frac{\rho(x_n,y_n)^{\alpha}}{D^\alpha} g_2\delta_{z_n}
	\end{equation*}
	where $\delta_x$ is the Dirac delta measure concentrated in $x$. Hence we can conclude that $\mu_n$ is a purely atomic measure with $card(supp(\mu_n))\le 3$.
\end{proof}
Let us remark that the set of finitely supported measure is dense in $M(K)$ with respect to the Kantorovich-Rubinstein norm, so the fact that $\mu$ could be a non-atomic measure does not contradict our result.
\subsubsection*{Acknowledgements} The first, third and fourth authors are members of the Gruppo Nazionale per l'Analisi Matematica, la Probabilit\'a e le loro Applicazioni (GNAMPA) of the Istituto Nazionale di Alta Matematica (INdAM), while the second author is member of the Gruppo Nazionale per il Calcolo Scientifico (GNCS) of the Istituto Nazionale di Alta Matematica (INdAM). The third author has been also supported by \lq\lq~Sostegno alla Ricerca Locale~\rq\rq , Universit\'a degli studi di Napoli \lq\lq~Parthenope~\rq\rq. 

\bibliographystyle{plain}
\bibliography{references}

\begin{thebibliography}{10}

\bibitem{gigli2011user}
L.~Ambrosio and N.~Gigli.
\newblock {\em A User's Guide to Optimal Transport}, pages 1--155.
\newblock Springer Berlin Heidelberg, Berlin, Heidelberg, 2013.

\bibitem{ambrosiotilli}
L.~Ambrosio and P.~Tilli.
\newblock {\em Topics on analysis in metric spaces}, volume~25 of {\em Oxford
  Lecture Series in Mathematics and its Applications}.
\newblock Oxford University Press, Oxford, 2004.

\bibitem{angrisani2019orlicz}
F.~Angrisani, G.~Ascione, and G.~Manzo.
\newblock Orlicz spaces with a o--{O} type structure.
\newblock {\em Ricerche di Matematica}, pages 1--17, 2019.

\bibitem{bade1987amenability}
W.G. Bade, PC~Curtis~Jr, and HG~Dales.
\newblock Amenability and weak amenability for {B}eurling and {L}ipschitz
  algebras.
\newblock {\em Proceedings of the London Mathematical Society}, 3(2):359--377,
  1987.

\bibitem{berninger2003lipschitz}
H.~Berninger and D.~Werner.
\newblock Lipschitz spaces and {M}-ideals.
\newblock {\em Extracta mathematicae}, 18(1):33--56, 2003.

\bibitem{bourgain2015new}
J.~Bourgain, H.~Brezis, and P.~Mironescu.
\newblock A new function space and applications.
\newblock {\em J. Eur. Math. Soc.(JEMS) 6}, 17(9):2083--2101, 2015.

\bibitem{brezis2010functional}
H.~Brezis.
\newblock {\em Functional analysis, {S}obolev spaces and partial differential
  equations}.
\newblock Springer Science \& Business Media, 2010.

\bibitem{coifman2006analyse}
R.~R. Coifman and G.~Weiss.
\newblock {\em Analyse harmonique non-commutative sur certains espaces
  homog{\`e}nes: {\'e}tude de certaines int{\'e}grales singuli{\`e}res}, volume
  242.
\newblock Springer, 2006.

\bibitem{de1961banach}
K.~De~Leeuw.
\newblock Banach spaces of {L}ipschitz functions.
\newblock {\em Studia Mathematica}, 1(21):55--66, 1961.

\bibitem{d2019atomic}
L.~D'Onofrio, L.~Greco, K.-M. Perfekt, C.~Sbordone, and R.~Schiattarella.
\newblock Atomic decompositions, two stars theorems, and distances for the
  {B}ourgain-{B}rezis-{M}ironescu space and other big spaces.
\newblock {\em arXiv preprint arXiv:1907.06380}, 2019.

\bibitem{DSSJEPE}
L.~D'Onofrio, C.~Sbordone, and R.~Schiattarella.
\newblock Duality and distance formulas in {B}anach function spaces.
\newblock {\em J. Elliptic Parabol. Equ.}, 5(1):1--23, 2019.

\bibitem{gogatishvili2010interpolation}
A.~Gogatishvili, P.~Koskela, and N.~Shanmugalingam.
\newblock Interpolation properties of {B}esov spaces defined on metric spaces.
\newblock {\em Mathematische Nachrichten}, 283(2):215--231, 2010.

\bibitem{Grigoryan}
A.~Grigor'yan and L.~Liu.
\newblock Heat kernel and {L}ipschitz--{B}esov spaces.
\newblock In {\em Forum Mathematicum}, volume~27, pages 3567--3613. De Gruyter,
  2015.

\bibitem{hajlasz2003sobolev}
P.~Haj\l{}asz.
\newblock {\em Sobolev spaces on metric-measure spaces}, volume 338 of {\em
  Contemp. Math.}, pages 173--218.
\newblock Amer. Math. Soc., Providence, RI, 2003.

\bibitem{Hanin}
L.~G. Hanin.
\newblock Kantorovich-{R}ubinstein norm and its application in the theory of
  {L}ipschitz spaces.
\newblock {\em Proceedings of the American Mathematical Society},
  115(2):345--352, 1992.

\bibitem{harmand2006m}
P.~Harmand, D.~Werner, and W.~Werner.
\newblock {\em M-ideals in {B}anach spaces and {B}anach algebras}.
\newblock Springer, 2006.

\bibitem{MR3262204}
R.~Hurri-Syrj\"{a}nen, N.~Marola, and A.~V. V\"{a}h\"{a}kangas.
\newblock Aspects of local-to-global results.
\newblock {\em Bull. Lond. Math. Soc.}, 46(5):1032--1042, 2014.

\bibitem{kantorovich1957functional}
L.~V. Kantorovich and G.~S. Rubinshtein.
\newblock On a functional space and certain extremum problems.
\newblock In {\em Doklady Akademii Nauk}, volume 115, pages 1058--1061. Russian
  Academy of Sciences, 1957.

\bibitem{kantorovich1958space}
L.~V. Kantorovich and G.~S. Rubinstein.
\newblock On a space of completely additive functions.
\newblock {\em Vestnik Leningrad. Univ}, 13(7):52--59, 1958.

\bibitem{kantorovich1942mass}
L.V. Kantorovich.
\newblock On mass transfer problem.
\newblock {\em Dokl. Acad. Nauk SSSR37}, pages 199--201, 1942.

\bibitem{karak2019measure}
N.~Karak.
\newblock Measure density and embeddings of {H}aj{\l}asz-{B}esov and
  {H}aj{\l}asz-{T}riebel-{L}izorkin spaces.
\newblock {\em Journal of Mathematical Analysis and Applications},
  475(1):966--984, 2019.

\bibitem{luukkainen1998every}
J.~Luukkainen and E.~Saksman.
\newblock Every complete doubling metric space carries a doubling measure.
\newblock {\em Proceedings of the American Mathematical Society},
  126(2):531--534, 1998.

\bibitem{Perfekt}
K.-M. Perfekt.
\newblock Duality and distance formulas in spaces defined by means of
  oscillation.
\newblock {\em Arkiv f{\"o}r matematik}, 51(2):345--361, 2013.

\bibitem{perfekt2015weak}
K.-M. Perfekt.
\newblock Weak compactness of operators acting on o--{O} type spaces.
\newblock {\em Bulletin of the London Mathematical Society}, 47(4):677--685,
  2015.

\bibitem{perfekt2017m}
K.-M. Perfekt.
\newblock On {M}-ideals and o--{O} type spaces.
\newblock {\em Mathematica Scandinavica}, 102(1):151--160, 2017.

\bibitem{wulbert1974representations}
D.~E. Wulbert.
\newblock Representations of the spaces of {L}ipschitz functions.
\newblock {\em Journal of Functional Analysis}, 15(1):45--55, 1974.

\end{thebibliography}
\end{document}